\documentclass[a4paper,12pt, reqno]{amsart}
\usepackage{a4wide}
\usepackage{amsthm}
\usepackage{amssymb}
\usepackage{palatino}
\usepackage{lipsum}
\usepackage{longtable}
\usepackage{enumitem}
\usepackage{hyperref}

\usepackage{rotating}
\newtheorem{thm}{Theorem}[section]
\newtheorem{cor}[thm]{Corollary}

\newtheorem{lem}[thm]{Lemma}

\theoremstyle{definition}

\theoremstyle{remark}

\numberwithin{equation}{section}

\begin{document}
\title[Some new results on negative polynomial Pell's equation]{Some new results on negative polynomial Pell's equation}
	\author{K. Anitha}
	\author{I. Mumtaj Fathima }
	\author{A R Vijayalakshmi }
	\address{Department of Mathematics, SRM IST Ramapuram, Chennai 600089, India}
	\address{Research Scholar, Department of Mathematics, Sri Venkateswara College of Engineering \\ Affiliated to Anna University, Sriperumbudur, Chennai 602117, India}
	\address{Department of Mathematics, Sri Venkateswara College of Engineering, Sriperumbudur, Chennai 602117, India}
	\email[K. Anitha]{subramanianitha@yahoo.com}
	\email[I. Mumtaj Fathima ]{tbm.fathima@gmail.com}
	\email[A R Vijayalakshmi ]{avijaya@svce.ac.in}
	\begin{abstract}
	We consider the negative polynomial Pell's equation 	$P^2(X)-D(X)Q^2(X)=-1$, where $D(X)\in \mathbb{Z}[X]$ be some fixed, monic, square-free, even degree polynomials. In this paper, we investigate the existence of polynomial solutions $P(X), \, Q(X)$ with integer coefficients.
	\end{abstract}

\maketitle
\section{Introduction}
The classical \textit{Pell's equation} is
\begin{equation}\label{p1}
	x^2-Dy^2=1,
\end{equation}
where $D$ is a square-free positive integer. Solving a Pell's equation for integers $x$ and $y$ is one of the classical problems in number theory. In $1768$, Lagrange proved that the equation \eqref{p1} has infinitely many solutions (\cite[vol. XXIII, p. 272]{lagf}, \cite[vol. XXIV, p. 236]{lags}). In fact, a classical result says that there exists a non-trivial solution $(x_0,y_0)$ is called a fundamental solution such that any other solution takes the form $(x_0+y_0\sqrt{D})^n, \, n\in \mathbb{Z}$.

\medskip

\noindent
On the other hand, the problem of solving a \textit{negative Pell's equation} has not been understood satisfactorily. It is an equation of the form
\begin{equation}\label{p2}
	x^2-Dy^2=-1,
\end{equation}
where $D$ is a square-free integer and $x, \, y$ are integer solutions. There is no solution for equation \eqref{p2} if $D$ is a negative integer and the length of the period in the continued fraction expansion of $\sqrt{D}$ is even. Nevertheless, if the length of the period in the continued fraction expansion of $\sqrt{D}$ is odd, then \eqref{p2} has infinitely many integer solutions \cite[Theorem 7.26]{stark}. Further, negative Pell's equation is not solvable for $D$ with prime divisor congruent to $3$ mod $4$ or $D$ is divisible by $4$. Moreover, \'{E}. Fouvry and J. Kl\"{u}ners \cite{Fouvry} gave the upper and lower bounds for the long-lasting conjecture on the asymptotic formulae for the number of square-free integers $D$ for which fundamental solution of the equation \eqref{p2}  has norm $-1$. Recently, the bound was further improved by P. Koymans and C. Pagano \cite{carlo}.

\medskip
\noindent
Similarly, we can consider the \textit{polynomial Pell's equation}
\begin{equation}\label{p3}
	P^2(X)-D(X)Q^2(X)= \pm 1,
\end{equation}
where $D(X)$ is a given fixed, square-free polynomial with integer coefficients and $P(X), Q(X)$ are its integer polynomial solutions.

In $1976$, Nathanson  \cite{nath} proved that when $D(X)=X^2+d\in \mathbb{Z}[X]$, the equation $P^2(X)-D(X)Q^2(X)=  1$ is solvable in $\mathbb{Z}[X]$ if and only if $d=\pm1, \pm2$. Moreover, such a polynomial solutions can be expressed in terms of Chebyshev polynomials \cite{pas}. In 2004, A. Dubickas and J. Steuding \cite{dubi} extended Nathanson's result for polynomials of the form $\displaystyle D(X)=X^{2k}+d\in \mathbb{Z}[X], \, k\in \mathbb{N}$. More precisely, they proved that the equation $P^2(X)-(X^{2k}+d)Q^2(X)=1$ is solvable in $\mathbb{Z}[X]$ if and only if $d\in\{\pm1, \, \pm2\}$.

\medskip
Despite many results on positive polynomial Pell's equation, there is no notable work on \textit{negative polynomial Pell's equation},
\begin{equation}\label{p4}
	P^2(X)-D(X)Q^2(X)=-1,
\end{equation}
where $D(X)$ be a fixed, even degree, square-free polynomial with integer coefficients and $P(X), \, Q(X)$ are its integer polynomial solutions. More precisely, we prove the following theorem:

\begin{thm}\label{main-thm}
	Let $d$ be an integer with $d\neq\pm1, \, \pm2$. Then the negative polynomial Pell's equation,
	\begin{equation}\label{p5}
		P^2(X)-(X^2+d)Q^2(X)=-1
	\end{equation}
	has no non-trivial solutions over $\mathbb{Z}[i]$.
\end{thm}

\begin{thm}\label{main-thm2}
	The equation \eqref{p5} has non-trivial polynomial solutions over $\mathbb{Z}$ if and only if $d=1$.
\end{thm}
\noindent We now generalize the Theorem \ref{main-thm2} and prove the following: 
\begin{thm}\label{gentheorem}
	The negative polynomial Pell's equation
	\begin{equation}\label{genpell}
		P^2(X)-(X^{2k}+d)Q^2(X)=-1,
	\end{equation}
	where $d\in \mathbb{Z}$ and $k\in \mathbb{N}$ has non-trivial solutions in $\mathbb{Z}[X]$ if and only if $d=1$.
\end{thm}
\noindent We need the following lemma to prove Theorem \ref{gentheorem}.
\begin{lem}\label{genlemma}
	Let $D(X)$ be a polynomial in $\mathbb{C}[X]$ with degree $2k$. Then the fundamental solutions $(U(X), V(X))$ in $\mathbb{C}[X]$ of equation \eqref{p4} satisfying $\operatorname{deg} U(X)=1/2\operatorname{deg} D(X)$ and $\operatorname{deg} V(X)=0$ is minimal.
\end{lem}
\begin{proof}
Firstly, let us consider $D(X)$ be a quadratic polynomial in $\mathbb{C}[X]$. We observe that the non-trivial solutions of \eqref{p4} exists only if $D(X)$ has distinct roots. Let $\gamma, \, \delta$ be the roots of $D(X)$. Then we write $D(X)=c(X-\gamma)(X-\delta), \, c\in \mathbb{C},\gamma\neq \delta$.\\ We set,
		\begin{center}
			$U(X)=\displaystyle\frac{2X-(\gamma+\delta)}{\sqrt{-1}(\gamma-\delta)}$; $V(X)=\displaystyle\frac{2}{\sqrt{-c}(\gamma-\delta)}$.
		\end{center}
For general case, we assume the contrary. Suppose that $\operatorname{deg} U(X)<1/2\operatorname{deg} D(X)$ and $\operatorname{deg} V(X)>0$. Since $\operatorname{deg} D(X)=2\operatorname{deg} P(X)-2\operatorname{deg} Q(X)$ and $\operatorname{deg} P(X)$ must be at least $1$ greater than the $\operatorname{deg} Q(X)$.\\
	Thus 
	$$
	\operatorname{deg} D(X)=2\operatorname{deg} U(X)-2\operatorname{deg} V(X)<\operatorname{deg} D(X)-2t,
	$$
	for some positive integer $t$. This completes the proof.
\end{proof}

\section{Proof of Theorem \ref{main-thm}}
We prove the theorem by contradiction. We first consider the equation \eqref{p5} as a polynomial over $\mathbb{Z}[i]$. We suppose that the equation \eqref{p5} has non-trivial solutions over $\mathbb{Z}[i]$. We choose a solution $P(X), \, Q(X)$ of \eqref{p5} with $\operatorname{deg} P(X)>0$ is minimal and we take a non-zero $d$ with $\mid d\mid\geq3$. We split the proof into two cases.

\medskip

\noindent
\textbf{Case(i):} If $d\neq-\alpha^2, \, \alpha \in \mathbb{Z}[i]$, then $X^2+d$ is irreducible over $\mathbb{Z}[i]$.
We now rewrite \eqref{p5} as,
\begin{equation}\label{p6}
	(P(X)+i)(P(X)-i)=(X^2+d)Q^2(X).
\end{equation}
Since $(X^2+d)$ is irreducible over $\mathbb{Z}[i]$ and $\mathbb{Z}[i]$ is a unique factorization domain, it divides one of the $(P(X)+i)$ or $(P(X)-i)$. We assume that $(X^2+d)$ divides $P(X)-i$ . Therefore,
\begin{equation*}
	P(X)-i = (X^2+d)P_1(X),
\end{equation*}
where $P_1(X)$ is a polynomial over $\mathbb{Z}[i]$.\\
 Then
\begin{equation*}
	P(X)-i + 2i = P(X)+i = (X^2+d)P_1(X)+2i.
\end{equation*}

\medskip

\noindent
On substituting into the equation \eqref{p6}, we have
\begin{equation*}
	P_1(X)((X^2+d)P_1(X)+2i)=Q^2(X).
\end{equation*}

\medskip

\noindent
Since the greatest common divisor of $P_1(X)$ and $(X^2+d)P_1(X)+2i$ is $1$ or $2$, we must obtain at least one of the following conditions:

\medskip
\begin{enumerate}
	\item $(X^2+d)P_1(X)+2i = P_{2}^2(X),$ \quad $P_1(X)=Q_{2}^2(X);$\\
	\item $(X^2+d)P_1(X)+2i = -P_{2}^2(X),$ \quad $P_1(X)=-Q_{2}^2(X);$ \\
	\item $(X^2+d)P_1(X)+2i = -iP_{2}^2(X),$ \quad $P_1(X)=iQ_{2}^2(X);$ \\
	\item $(X^2+d)P_1(X)+2i =  iP_{2}^2(X),$ \quad $P_1(X)=-iQ_{2}^2(X);$ \\
	\item $(X^2+d)P_1(X)+2i = 2P_{2}^2(X),$ \quad $P_1(X)=2Q_{2}^2(X);$ \\
	\item $(X^2+d)P_1(X)+2i = -2P_{2}^2(X),$ \quad $P_1(X)=-2Q_{2}^2(X);$ \\
	\item $(X^2+d)P_1(X)+2i = -2iP_{2}^2(X),$ \quad $P_1(X)=2iQ_{2}^2(X);$\\
	\item $(X^2+d)P_1(X)+2i = 2iP_{2}^2(X),$ \quad $P_1(X)=-2iQ_{2}^2(X).$
\end{enumerate}

\medskip

\medskip
As $P_2(X)$ is a polynomial over $\mathbb{Z}[i]$. We substitute $X=\sqrt{-d}$ in conditions $(1)-(8)$ and we see that the following possibilities are admissible:  $(r+s\sqrt{-d})^2=\pm 2i$ or $(r+s\sqrt{-d})^2=\pm2$ or $(r+s\sqrt{-d})^2=\pm i$ or $(r+s\sqrt{-d})^2=\pm 1$ for some $r, \, s \in \mathbb{Z}[i]$. We need the following arguments to sort out the impossible conditions.

\medskip

\noindent
We first consider that $(r+s\sqrt{-d})^2=\pm 2i$ and  $(r+s\sqrt{-d})^2=\pm i$. Substituting $r=x+iy, \, s=u+iv$, where $x, \, y, \, u, \, v\in \mathbb{Z}$, we have
\begin{equation*}
	(x+iy)^2-(u+iv)^{2}d+2i((x+iy)(u+iv))\sqrt{d}=\pm 2i, \, \pm i.
\end{equation*}
On equating real and imaginary parts, we get
\begin{align}
	x^2-y^2-(u^2-v^2)d-2\sqrt{d}(xv+yu) &=0, \label{p8} \\
	xy-uvd+(xu-vy)\sqrt{d} &=\pm 1, \, \pm 1/2.
\end{align}

\medskip

\noindent
By our choice of $d$, equation \eqref{p8} can be separated as rational and irrational parts,
\begin{equation*}
	x^2-y^2-(u^2-v^2)d=0.
\end{equation*}
This could be possible only when $d$ is a perfect square or $d=\pm 1$. This ends in a contradiction.

\medskip

\noindent
Now we will explore the equation $(r+s\sqrt{-d})^2=\pm 2$. As we proceeded before, we equate real and imaginary parts and we obtain
\begin{align}
	x^2-y^2-(u^2-v^2)d-2\sqrt{d}(xv+yu) &=\pm 2, \\
	xy-uvd+(xu-vy)\sqrt{d} &=0. \label{p10}
\end{align}

Again we repeat the same procedure as separating rational and irrational parts,
\begin{align}
	x^2-y^2-(u^2-v^2)d &=\pm 2,\\
	xv+yu &=0 \label{p12}, \\
	xy-uvd &=0 \label{p13},\\
	xu-vy &=0 \label{p14}.
\end{align}

\medskip

\noindent
By solving the simultaneous equations \eqref{p12} and \eqref{p14}, we get either $y=0$ or $u^2+v^2=0$. We first assume that $y=0$ and $x\neq0$ then $u=v=0$. Therefore $x=\pm \sqrt{2}$ or $\pm i\sqrt{2}$. Since $x$ is an integer, both can not be possible. On the other hand, if we assume both $x$ and $y$ are zero, then $uv=0$ (by using \eqref{p13}). Again a contradiction. Hence we conclude that $y$ should be a non-zero and  $u^2+v^2=0$. Here the only possibility is $u=v=0$. Thus we end with $x=0$ (by using \eqref{p10}) and the values of $y$ are $\pm \sqrt{2}$ or $\pm i\sqrt{2}$. This is again a contradiction.

\medskip

\noindent
Now we take $(r+s\sqrt{-d})^2=\pm1$. As we done in the previous arguments, we first deal with the equation,
\begin{equation}\label{p15}
	x^2-y^2-(u^2-v^2)d=1.
\end{equation}
There are two cases either $y=0$ or $u^2+v^2=0$ (by using \eqref{p12} and \eqref{p14}). At first, we suppose to consider both $x$ and $y$ are zero. Then we obtain $uv=0$ (by using \eqref{p13}). So we omit it. If we assume $y=0$ and $x\neq 0$, then $u=v$. Thus $x=\pm 1$ and the value of $r$ is $\pm 1$. On the other side, if $u^2+v^2=0$ then $u=v=0$. Therefore value of $s=0$.
\medskip

\noindent
Finally, we consider the equation
\begin{equation*}
	x^2-y^2-(u^2-v^2)d=-1.
\end{equation*}
Again by the same procedure as we deal with the equation \eqref{p15}, we end with $y=\pm 1$ and $u=v=x=0$. Thus $r=\pm i, \, s=0$. Among eight conditions, only $(7)$ and $(8)$ are possible. We now rewrite the condition $(7)$ as,
\begin{equation}\label{p17}
	P_{2}^2(X)-(X^2+d)(iQ_{2}(X))^2=-1
\end{equation}
and condition $(8)$ as,
\begin{equation}\label{p18}
	(iP_{2}(X))^2-(X^2+d)Q_{2}^2(X)=-1.
\end{equation}

\medskip

\noindent
But in both equations \eqref{p17} and \eqref{p18}, $2 \operatorname{deg}(P_2(X))=2+ \operatorname{deg} (P_1(X))= \operatorname{deg}(P(X))$. It leads to a contradiction on minimality of $\operatorname{deg} (P(X))$. Therefore equation \eqref{p5} has no non-trivial solutions, if $d\, (\neq \pm 1, \, \pm 2)$ and also $d\neq -\alpha^2, \, \alpha\in \mathbb{Z}[i]$.

\medskip

\noindent
$\textbf{Case(ii):}$
\medskip
Let $d=-\alpha^2, \, \alpha$ be a non-unit in $\mathbb{Z}[i]$ and $N(\alpha)>2$. The constant term of the solution polynomials $P(X)$ and $Q(X)$ are $\pm i, \, 0 $
respectively. Suppose that $P(0)=i$. Then $P(X)=i+XP_1(X)$ and $Q(X)=XQ_1(X)$. We substitute $P(X), \, Q(X)$ into equation \eqref{p5} and we obtain,
\begin{equation}\label{p19}
	P_1(X)(XP_1(X)+2i)=X(X^2-\alpha^2)Q_{1}^2(X).
\end{equation}
Since $P_1(X)$ is a polynomial without constant term, we write $P_1(X)=XP_2(X)$. Now we rewrite \eqref{p19} as,
\begin{equation}\label{p20}
	P_2(X)(X^2P_2(X)+2i)=(X^2-\alpha^2)Q_{1}^2(X).
\end{equation}
We suppose that $X\pm \alpha$ divides $X^2P_2(X)+2i$. Then we put $X=\mp \alpha$ and we get $\alpha^2P_2(\mp \alpha)=-2i$. Thus $\alpha^2$ divides $2i$. Since $N(\alpha)>2$, this is not possible. Therefore both $X+\alpha$ and $X-\alpha$ should divide $P_2(X)$. We can say $P_2(X)=(X^2-\alpha^2)P_3(X)$. On substituting in \eqref{p20}, we obtain,
$$
P_3(X)(X^2(X^2-\alpha^2)P_3(X)+2i)=Q_{1}^2(X).
$$
The greatest common divisor of $P_3(X)$ and $X^2(X^2-\alpha^2)P_3(X)+2i$ is $1$ or $2$. Again we repeat the same procedure as in \textit{case(i)}. This completes the proof of Theorem \ref{main-thm}.

\section{Continued fraction expansion of $\sqrt{D(X)}$}
In this section, we use the same technique which is used for irrational $\sqrt{D}$ in \cite{niven}.\\
The continued fraction expansion of $\sqrt{D(X)}$ is of the form $$
[a_0(X),\overline{\displaystyle a_1(X),a_2(X),\dots,a_{r-1}(X), 2a_0(X)}]
$$ 
with convergents $H_n(X)/K_n(X)$ and $a_i(X)$ be a non-constant polynomial in $\mathbb{Z}[X]$. Let $r$ be the length of the shortest period in the continued fraction expansion of $\sqrt{D(X)}$. \\ We define 
$$
\zeta_0(X)=\frac{M_0(X)+\sqrt{D(X)}}{N_0(X)}
$$
with $N_0(X)=1$ and $M_0(X)=0$.\\ In general, we define
\begin{align*}
a_i(X)& =[\zeta_i(X)],\\
\zeta_i(X)& =\frac{M_i(X)+\sqrt{D(X)}}{N_i(X)},\\
M_{i+1}(X)& = a_i(X)N_i(X)-M_i(X),\\
N_{i+1}(X) &=\frac{D(X)-M^2_{i+1}(X)}{N_i(X)},
\end{align*} 
where $[.]$ denotes the rational part of the polynomial in terms of $X$. Since $r$ be the length of the period, we write $\zeta_0=\zeta_r=\zeta_{2r}=\dots$. Thus for all $j\geq0$ we write
$$
\frac{M_{jr}(X)+\sqrt{D(X)}}{N_{jr}(X)}=\zeta_{jr}(X)=\zeta_0(X)=\frac{M_{0}(X)+\sqrt{D(X)}}{N_{0}(X)}.
$$
\begin{thm}
	If $D(X)$ is a square-free polynomial in $\mathbb{Z}[X]$ with period length $r$, then $H_{n}(X)^2-D(X)K_{n}^2(X)=(-1)^{n-1}N_{n+1}(X)$.
	\end{thm}
\begin{proof}
The well-known classical result (\cite[Theorem 7.3]{niven}) says that,
\begin{align*}
	\zeta_0(X) &=[a_0(X), a_1(X), a_2(X), \dots, a_n(X), \zeta_{n+1}(X)]\\
	&= \frac{\zeta_{n+1}(X)H_n(X)+H_{n-1}(X)}{{\zeta_{n+1}(X)K_n(X)}+K_{n-1}(X)}\\
	&=\frac{\bigg(\frac{M_{n+1}(X)+\sqrt{D(X)}}{N_{n+1}(X)}\bigg)H_n(X)+H_{n-1}(X)}{\bigg(\frac{M_{n+1}(X)+\sqrt{D(X)}}{N_{n+1}(X)}\bigg)K_n(X)+K_{n-1}(X)}\\
\displaystyle\sqrt{D(X)} &= \frac{\bigg(M_{n+1}(X)+\sqrt{D(X)}\bigg)H_n(X)+H_{n-1}(X)N_{n+1}(X)}{\bigg(M_{n+1}(X)+\sqrt{D(X)}\bigg)K_n(X)+K_{n-1}(X)N_{n+1}(X)}.
\end{align*}
We separate it as rational and irrational part and equate each part into zero.
\begin{align}
	-M_{n+1}(X) H_n(X)+K_n(X)D(X)-H_{n-1}(X)N_{n+1}(X) &=0\label{x}\\
	M_{n+1}(X) K_{n}(X)+N_{n+1}(X)K_{n-1}(X)-H_n(X) &=0.\label{y}
\end{align}
We eliminate $M_{n+1}(X)$ from above equations \eqref{x} and \eqref{y}. Then we write
$$
	H_n^2(X)-D(X)K_n^2(X) =(H_n(X)K_{n-1}(X)-K_n(X)H_{n-1}(X))N_{n+1}(X)
$$
Then by using the result $H_n(X)K_{n-1}(X)-K_n(X)H_{n-1}(X)=(-1)^{n-1}$ \cite[Theorem 7.5]{niven}, we obtain
	\begin{equation}\label{contlength}
H_n^2(X)-D(X)K_n^2(X) =(-1)^{n-1}N_{n+1}(X).
	\end{equation}
This completes the proof.
\end{proof}
\begin{cor}\label{corollary}
	Let $r$ be the length of the period in the continued fraction expansion of $\sqrt{D(X)}$. Then we have for $n\geq0$ equation \eqref{contlength} becomes,
$$
		H_{nr-1}^2(X)-D(X)K_{nr-1}^2(X) =(-1)^{nr}N_{nr}(X)	=(-1)^{nr}.
$$

\end{cor}
\begin{proof}
	We replace $n$ by $nr-1$ in equation \eqref{contlength}. 
	$$
	H_{nr-1}^2(X)-D(X)K_{nr-1}^2(X)=(-1)^{nr}N_{nr}(X)=(-1)^{nr}N_{0}(X)=(-1)^{nr}.
	$$
\end{proof}
\section{The $ABC$ conjecture for polynomials (W. W. Stothers and R. C. Mason)}
W. W. Stothers \cite{stothers} and R. C. Mason \cite{mason} independently proved the $ABC$ conjecture for polynomials.

Let $n_{0}(P(X))$ denotes the number of distinct complex zeros of a polynomial $P(X)$(which does not vanish identically). If $A,B,C$ are coprime polynomials over $\mathbb{C}$, not all constant polynomials satisfy $A+B=C$ then
\begin{equation}\label{abcpoly}
	\max\{\operatorname{deg} A,\operatorname{deg}  B, \operatorname{deg} C\}<n_0(ABC).
\end{equation}

In $1984$, J. H. Silverman \cite{silver} using Riemann-Hurwitz formula and gave the different proof. In $2000$, N. Snyder \cite{snyder} gave the slight alternate proof of the Stothers-Mason theorem.  The connection between the inequality \eqref{abcpoly} and Fermat's last theorem for polynomials can be found in Lang's survey article \cite{lang}. The $ABC$ conjecture for polynomials have notable applications to polynomial Pell's equation.
The following lemma is inspired by the result in \cite[Theorem 1]{dubi}.
\begin{lem}
	If $n_0(D(X))$, where $D(X)\in \mathbb{C}[X]$ is less than or equal to $1/2\operatorname{deg} D(X)$, then the negative polynomial Pell's equation \eqref{p4} has no non-trivial solutions in $\mathbb{C}[X]$.
\end{lem}
\begin{proof}
	We consider $A=P^2(X), \, B=-D(X)Q^2(X), \, C=-1$.\\
	We note that $\max\{\operatorname{deg} A, \operatorname{deg} B, \operatorname{deg} C\}=\operatorname{deg} B$ and $n_0(P(X))\leq \operatorname{deg} P(X), \, n_0(Q(X))\leq \operatorname{deg} Q(X)$.\\
	 By using $ABC$ conjecture for polynomials, we write
	\begin{align*}
		\operatorname{deg} D(X)Q^2(X) & < n_0(P^2(X)D(X)Q^2(X)) \\
		& = n_0(P(X)D(X)Q(X)) \\
		\operatorname{deg} D(X) & < n_0(P(X))+n_0(D(X))+n_0(Q(X))-2\operatorname{deg} Q(X) \\
		\operatorname{deg} D(X) & < \operatorname{deg} P(X)-\operatorname{deg} Q(X)+n_0(D(X))\\
		1/2\operatorname{deg} D(X) & < n_0(D(X)).
	\end{align*}
	This completes the proof.
\end{proof}

\section{Proof of Theorem \ref{gentheorem}}
	We use the method of continued fraction expansion of $\sqrt{X^{2k}+d}, \,d\in \mathbb{Z}$.\\
	i. e. ,
	$$
	\sqrt{X^{2k}+d}=[X^k, \overline{2X^k/d,2X^k}].
	$$
By using Lemma \ref{genlemma}, the fundamental solution over $\mathbb{C}$ is $\big(\frac{X^k}{\sqrt{d}}, \frac{1}{\sqrt{d}}\big), \, d\in\mathbb{Z}$. The integer polynomial solution is possible only for odd periodic lengths.\\
	Thus
	\begin{align*}
		\displaystyle\left(\frac{X^k+\sqrt{X^{2k} +d}}{\sqrt{d}}\right)^{2n-1}& =\frac{1}{d^{(2n-1)/2}}\left(X^k+\sqrt{X^{2k}+d}\right)^{2n-1} \\
		& = P_{2n-1}(X)+\sqrt{X^{2k} +d}Q_{2n-1}(X), \, n\in \mathbb{N}
	\end{align*}
	We expand the powers as we done in the Theorem \ref{main-thm2}. Thus to show the existence of non-trivial solutions in $\mathbb{Z}[X]$ for the negative polynomial Pell's equation \eqref{genpell}, it is enough to show that the leading coefficient of $P_{2n-1}(X)$ is an integer. Thus the coefficient of $X^{k(2n-1)}$ in $P_{2n-1}(X)$ is
	\begin{center}
		$\displaystyle\frac{1}{d^{(2n-1)/2}}\left(1+\binom{2n-1}{2}+\binom{2n-1}{4}+\dots \right)=\frac{2^{(2n-2)}}{d^{(2n-1)/2}}$.
	\end{center}
	The integer solutions are exist if and only if d=1. This completes the proof of the theorem.
	
\noindent The following theorems are some of other generalization of the Theorem \ref{main-thm2}.
\begin{thm}\label{pell1}
	The negative polynomial Pell's equation 
\begin{equation}
		P^2(X)-(X^{2k}+aX+b)Q^2(X)=-1,
\end{equation}
	where $a, \, b\in\mathbb{Z}$ has no non-trivial solutions in $\mathbb{Z}[X]$.
\end{thm}
\begin{thm}
		The negative polynomial Pell's equation 
	\begin{equation}\label{pell2}
			P^2(X)-(X^{2k}+aX^k+b)Q^2(X)=-1,
	\end{equation}	
where $a, \, b\in\mathbb{Z}$ and $k\in\mathbb{N}$ has no non-trivial solutions in $\mathbb{Z}[X]$.
\end{thm} 
Since the length of the period in the continued fraction expansions of both $\sqrt{X^{2k}+aX+b}$ and $\sqrt{X^{2k}+aX^k+b}$ is $2$, then by Corollary \ref{corollary} we can say both the negative polynomial Pell's equations \eqref{pell1} and \eqref{pell2} have no non-trivial solutions in $\mathbb{Z}[X]$.

\section*{Acknowledgement} 
We express our sincere thanks to Dr. Carlo Pagano, Max Planck Institute for Mathematics, Bonn, for his valuable suggestions in the earlier version of this paper and for pointing out  some corrections in reference \cite{Fouvry}.
The author I. Mumtaj Fathima would like to express her gratitude to Maulana Azad National Fellowship for minority students, UGC. This research work is supported by MANF-2015-17-TAM-56982, University Grants Commission (UGC), Government of India.

\end{document}